\documentclass[12pt,reqno]{amsart}
\usepackage{geometry}                
\usepackage{subfig}
\usepackage{amsmath,amssymb}
\usepackage{graphicx}
\usepackage{amssymb,latexsym, amsmath}
\usepackage{amsfonts,amsthm}
\usepackage{amscd}
\usepackage{mathrsfs}
\usepackage{hyperref}
\usepackage{eucal}
\usepackage{multicol}
\usepackage{epstopdf}
\usepackage{amsgen}
\usepackage{xspace}
\usepackage{verbatim}
\usepackage{stmaryrd}
\usepackage{graphicx}
\usepackage{caption}

\newcommand{\gw}{\Omega}
\newcommand{\ap}{\alpha}
\newcommand{\ga}{\gamma}

\newcommand{\gb}{\beta}
\newcommand{\G}{\Gamma}
\newcommand{\gl}{\lambda}

\newcommand{\ms}{\mathscr}

\newcommand{\ve}{\varepsilon}
\newcommand{\dca}{D_c^\alpha}

\newcommand{\beq}{\begin{equation}}
\newcommand{\eeq}{\end{equation}}
\newcommand{\bea}{\begin{align}}
\newcommand{\eea}{\end{align}}
\newcommand{\bthm}{\begin{theorem}}
\newcommand{\ethm}{\end{theorem}}
\newcommand{\bpr}{\begin{proof}}
\newcommand{\epr}{\end{proof}}
\newcommand{\bcl}{\begin{corollary}}
\newcommand{\ecl}{\end{corollary}}
\newcommand{\bpn}{\begin{proposition}}
\newcommand{\epn}{\end{proposition}}
\newcommand{\bre}{\begin{remark}}
\newcommand{\ere}{\end{remark}}
\newcommand{\bdf}{\begin{definition}}
\newcommand{\edf}{\end{definition}}
\newcommand{\bss}{\begin{align*}}
\newcommand{\ess}{\end{align*}}

\newcommand{\bl}{\label}

\newtheorem{theorem}{Theorem}[section]
\newtheorem{corollary}[theorem]{Corollary}

\newtheorem{lemma}[theorem]{Lemma}
\newtheorem{proposition}[theorem]{Proposition}

\theoremstyle{definition}
\newtheorem{definition}[theorem]{Definition}
\theoremstyle{remark}
\newtheorem{remark}{Remark}

\numberwithin{equation}{section}

\begin{document}

\title[Synchronization of Hodgkin-Huxley-Wilson Neural Networks]{Global Dynamics and Synchronization of Hodgkin-Huxley-Wilson Neural Networks}

\author[Y. You]{Yuncheng You}
\address{University of South Florida, Tampa, FL 33620, USA}
\email{yygwmp@gmail.com}

\thanks{}

\subjclass[2020]{34D06, 34A08, 37N25, 92B20, 92C20}

\date{\today}


\keywords{Hodgkin-Huxley-Wilson neural network, synchronization, global dynamics, coupling threshold condition, time-fractional memristive neural network.}

\begin{abstract} 
Hodgkin-Huxley equations as a monumental breakthrough in neuroscience theory had been distilled into many simplified models to study, but this mathematical model itself was not fully investigated in terms of global and asymptotic dynamics due to its strong nonlinearity and higher dimensionality. In this paper a new model called Hodgkin-Huxley-Wilson neural networks is proposed and investigated. This model captures the essential features of the nonlinearity and the conductances of two dominant ionic current channels coupled with the neuron membrane voltage in gated firing and bursting functionality of biological neural networks in the original Hodgkin-Huxley model. Through the uniform and sharp \emph{a priori} estimates purely by mathematical hard analysis on the solutions of the model equations and then the interneuron differencing equations of gap functions, it is rigorously proved that global solution dynamics are robustly dissipative with a sharp ultimate bound and that exponential synchronization of the Hodgkin-Huxley-Wilson neural networks at a uniform convergence rate occurs if the interneuron coupling strength satisfies an explicitly computable threshold condition. The synchronization result with fractional-power convergence rate is also proved for time-fractional memristive Hodgkin-Huxley-Wilson neural networks. 
\end{abstract}

\maketitle
 
\section{\textbf{Introduction}}

All biological neurons especially human neuron networks are extremely complex with intricate behavior and functions. To study the neuronal dynamics and brain-like spiking neural networks, one can build mathematical models characterizing the key variables of neurons, which are the voltage potential of neuron cell membrane and the multiple ion currents. The primary mathematical model of biological neurons is the Hodgkin-Huxley equations initiated in \cite{HH} and four more papers \cite{HH1, HH2, HH3, HHB} based on plausibly equivalent electric circuits and large amount of data from experiments by the British physiologists and Nobel laureates (1963) Alan Lloyd Hodgkin and Andrew Fielding Huxley in 1952 based on their study of the giant axon of squid. Hodgkin-Huxley equations laid down a theoretical foundation for neuroscience \cite{KS, ET} as well as vast researches of brain-like spiking neural networks.

The original Hodgkin-Huxley neuron model consists of four coupled nonlinear ordinary differential equations describing how the membrane potential $V(t)$ of a neuron cell changes and governed by ions channel gating of three variables $ x = (m, n, h)$:
\beq \bl{hhe}
	\begin{split}
	C_m \frac{dV}{dt} &= I -g_{N_a} m^3 h (V - E_{Na}) - g_K n^4 (V - E_k) - g_L (V - E_L),   \\
	\frac{dm}{dt} &= \ap_m (V) (1 - m) - \gb_m (V) m,  \\
	\frac{dn}{dt} &= \ap_n (V) (1 - n) - \gb_n (V) n,  \\
	\frac{dh}{dt} &= \ap_h (V) (1 - h) - \gb_h (V) h,  
	\end{split}
\eeq
where the gating current variables of sodium ions $N_a^+$ denoted by $m(V(t))$, potassium ions $K^+$ denoted by $n(V(t))$, and the other (leakage) ions denoted by $h(V(t))$ can also be called sodium, potassium, and leakage activation, respectively. Total external current is denoted by $I$. The positive constant $C_m$ is the membrane capacitance, without loss of generality which is usually taken as $1\mu F/cm^2$.

For the original model of Hodgkin-Huxley equations \cite{KS, ET}, the ionic channel conductance parameters are $g_{N_a} = 120,  g_K = 36,  g_L = 0.3$ with the physics unit of electrical conductivity $ms/cm^2$ ($ms =$ milliSiemens). The Nernst equilibrium potentials of the voltage-gated channels are $E_{N_a} = 50, E_K = -70, E_L = - 10$ in $mV$.

The voltage-dependent rates $(\ap(V), \gb (V))$ of the ionic channels reflect the speed of depolarization and repolarization gating. They were originally set by experiment data and curve-fitting diagrams using combination of the exponential negative functions $\exp (- \gl (V- V_e))$. It enables to execute quadratic nonlinear polynomials as  appropriate approximation in Wilson model for human and mammalian neurons \cite{Wi} and in the new model of biological neural networks here in this work.. 

As suggested in the original Hodgkin-Huxley model,  the gate transition rates of neuron bursting from closed to open and vice versa can be presented in terms of probability so that the equation for an ion channel $x(V)$ can take the following form,
\beq \bl{mta}
	\frac{dx}{dt} = \ap_x (V) (1 - x) - \gb_x (V) x = (x_\infty (V) - x)/\tau_x (V)
\eeq
where 
\beq \bl{mta}
	x_\infty (V) = \frac{\ap_x(V)}{\ap_x(V) + \gb_x(V)} = \frac{1}{1 + \exp{(- \gl_s (V - V_f))}}, \;\,  \tau_x (V) = \frac{1}{\ap_x(V) + \gb_x(V)}.
\eeq
The sigmoidal function $x_\infty (V)$ has two parameters: $V_{f}$ is the half-activated voltage and the slope factor $\gl_s$ determines the sensitivity of this ion gating to voltage change. 

Hodgkin-Huxley neuron model is a highly nonlinear four-dimensional system of ordinary differential equations. Many simplified neuron dynamics models have been proposed and extensively studied: Two-dimensional FitzHugh-Nagumo model \cite{FHN}, three-dimensional Hindmarsh-Rose model \cite{HR}, Morris-Lecar model \cite{ML}, Wilson model \cite{Wi}, Izhikevich model \cite{Zh}, fractional-order Hodgkin-Huxley neuron model, etc.

The early explorations of the mathematical aspects and qualitative behavior of solutions for the Hodgkin-Huxley equations can be found in \cite{ET, Cr, Ev, FY} and references therein. The firing, bursting, or chaotic properties of these neuron models were studied mainly by local stability analysis, bifurcation analysis, and numerical simulations \cite{JJ, GO, WCF, BVL}.

Since the axon of neuron cell is a long branch for propagation of bioelectrical signals and neurons are immersed in aqueous chemical solutions with diffusive charged ions, it is also meaningful to consider distributed PDE models with the spacial variable  involved in neuron dynamics. In recent decade, active and fruitful researches emerged on global and asymptotic dynamics of diffusive Hindmarsh-Rose equations and FitzHugh-Nagumo equations, cf. \cite{LSY, PYS, CY3} and references therein.

Synchronization dynamics is a significant topic about cognitive functionality and signals processing in regard to biological neural networks as well as artificial neural networks \cite{AB, BP, HJ, WD}. For diffusive Hindmarsh-Rose neural networks \cite{Y1, Y2}, partly diffusive FitzHugh-Nagumo neural networks \cite{YTT}, and reaction-diffusion neural networks \cite{YT} with memristive synapses, it has been analytically proved that exponential synchronization will be triggered if the network coupling strength satisfies a threshold condition. 

In view of the limited and sporadically published research results cf. \cite{CPF, BB, PL, BAA, AO, SS, WLZ} on synchronization dynamics of the primary Hodgkin-Huxley model of neuron ensembles and neural networks, it is seen that all the past and current study approaches are essentially or solely based on numerical experiments and simulations, occasionally with the linear matrix spectrum and local stability analysis.  

It remains widely open to make a substantial breakthrough and establish a rigorous analytic theory on global dynamics and asymptotic synchronization for the highly nonlinear Hodgkin-Huxley equations and Hodgkin-Huxley neural networks.

In this paper we shall formulate and investigate a new model of biological neural networks, which we call Hodgkin-Huxley-Wilson (briefly HHW) neural networks. This new model is an adaptation of the Hodgkin-Huxley equations partly adopted the quadratic nonlinear approximation from the Wilson model. It is different from the FitzHugh-Nagumo equations and Hindmarsh-Rose equations, which are conductance-independent neuron models without fully characterizing the effect of the flux-changing rates of the firing activation and recovery current channels in the original Hodgkin-Huxley model equations \cite{VMN, Qi, XJ}.

We shall analyze the global solution dynamics of the model equations and rigorously prove the complete exponential synchronization of the Hodgkin-Huxley-Wilson neural networks if the network coupling strength satisfies an explicitly computable threshold condition. We will also generalize the main result of complete synchronization but with fractional-power convergence rate to fractional memristive Hodgkin-Huxley-Wilson neural networks presented by Caputo-fractional differential equations. 

The main contribution of this paper is twofold. The established and explicitly computable threshold condition and the exponential convergence rate for complete synchronization of this type HHW neural networks is a substantial breakthrough in the research area of biological neural networks. Besides the theoretical proof approach, methodology, and technicalities shown in this paper provide a new and rigorous template other than numerical simulations to study asymptotic dynamics and synchronization topics for various mathematical models of either biological or artificial neural networks through effective \emph{a priori} estimates with mathematical hard analysis and differential inequality tools. 

The rest of the paper is organized as follows. In Section 2 we formulate Hodgkin-Huxley-Wilson neural networks in a mathematical framework and provide the quantitative definition of complete synchronization together with some preliminaries. In Section 3 we show that global dynamics of the solutions of the model equations is dissipative in terms of the existence of an absorbing set and the global attractor. In Section 4 we shall prove the main result on the exponential synchronization of the Hodgkin-Huxley-Wilson neural networks through \emph{a priori} uniform estimates of the gap functions, which satisfy the interneuron differencing equations, leveraged by analysis of differential inequalities. In Section 5 we extend the synchronization result to Caputo fractional memristive Hodgkin-Huxley-Wilson neural networks. Section 6 will be conclusions to summarize the contribution with future outlooks.

\section{\textbf{Formulation of Hodgkin-Huxley-Wilson Neural Networks}}

Consider a mathematical model of biological neural networks composed of $n$ coupled neurons, denoted by $\mathcal{NW} = \{\mathcal{N}_i: 1 \leq i \leq n\}$, where $n \geq 2$ is a positive integer. Each neuron $\mathcal{N}_i$ is presented by the following Hodgkin-Huxley-Wilson equations: 
\beq \bl{Meq}
	\begin{split}
	\frac{dV_i}{d t} & = - m_\infty (V_i)(V_i - E_{N_a}) - g_K R_i(V_i) (V_i - E_K) + J + \sum_{j=1}^n P(V_j - V_i),   \\
	\frac{dR_i}{d t} & = \frac{1}{\tau_K} \left(-R_i + R_\infty (V_i) \right),  \quad \text{for} \;\; t > 0, \;\; 1 \leq i \leq n.  
	\end{split} 
\eeq
Here the exponentially nonlinear conductances of the sodium activation channel in the original Hodgkin-Huxley model is well approximated by quadratic polynomial and the potassium recovery channel takes a sigmoidal form:
\beq \bl{mR}
	m_\infty (s) = a_0 + a_1 s + a_2 s^2,   \quad   R_\infty (s) =\frac{H}{1 + e^{- \gl (s - E_K)}},  \quad  s \in \mathbb{R}.
\eeq
The variable $V_i(t)$ is the electrical voltage on the membrane of the $i$th neuron $\mathcal{N}_i$, for $1 \leq i \leq n$. The variable $R_i(t)$ as a composed function $R_i(V_i(t))$ is the recovery $K^+$ ionic current in the $i$th neuron. The initial states of the ODE system \eqref{Meq} in the state space $\mathbb{R}^{2n}$ will be denoted by 
\beq \bl{inc}
	V^0 = \{V_i^0 = V_i(0): 1 \leq i \leq n\} \in \mathbb{R}^n,  \, \;  R^0 = \{R_i^0 = R_i (0): 1 \leq i \leq n\} \in \mathbb{R}^n.
\eeq

We make the following generic assumption on the parameters in the model \eqref{Meq} of Hodgkin-Huxley-Wilson neural networks:
\beq \bl{Asp}
	E_{N_a}, \, g_K, \, P, \, a_0, \, a_2, \, \gl,  \, H, \, \tau_K \in \mathbb{R}^+ = (0, \infty), \quad  E_K < 0, \;\; \text{and} \quad  J, \, a_1 \in \mathbb{R}. 
\eeq
In the original paper by Wilson \cite{Wi}, the above parameters were set based on the scaled voltage unit of $mV/100$: $a_0 = 17.8, a_1 = 47.6, a_2 = 33.8, g_K = 26 ms/cm^2, E_{N_a} = 0.5, E_K = -0.95, H = 1, \tau_K = 4.2\, ms$.

Here are some explanations. The main equation of the membrane potential $V_i (t)$ in this model \eqref{Meq} depends on the two dominant ionic currents of $N_a^+$ and $K^+$ and a stimulating current $J$. The coefficient $P > 0$ is called the network coupling strength. The sodium $N_a^+$ ionic channel is called activation and its conductance is set by $m_\infty (V_i)$ sufficiently fast \cite{Rz}. The second equation in \eqref{Meq} describes the potassium $K^+$ gating channel also called the recovery channel, which depends on the equilibrium state $R_\infty (V_i)$ and the time constant $\tau_K$. Note that $R_\infty (V_i)$ may take different nonlinear forms and $\tau_K$ may also be voltage-dependent. 

\begin{definition} \bl{Def}
	A biological neural network presented by a mathematical model of ordinary differential equations in a state space $Z$ such as the Hodgkin-Huxley-Wilson neural networks $\mathcal{NW}$ is said to be completely synchronizable if there exists a sufficient threshold condition on the network coupling strength parameter $\mathcal{P}$ such that if this threshold condition is satisfied then  the network synchronizing degree 
\beq \bl{gap}
	\deg_s \,(\mathcal{NW})= \sup_{(V^0, R^0) \in Z} \left\{ \max_{1 \leq i < j \leq n} \left\{\limsup_{t \to \infty} \, ( |V_i (t) - V_j(t)| + |R_i (t) - R_j(t)| )\right\}\right\} = 0,
\eeq
If the limsup convergence in \eqref{gap} admits a uniform exponential convergence rate, then it is called complete exponential synchronization.
\end{definition}

As preliminaries for global dynamics of differential equations \cite{CV} in the rest sections, a general dynamical system only for time $t \geq 0$ can be called semiflow. For a semiflow denoted by  $\{S(t)\}_{t \geq 0}$ on a Banach space $\ms{X}$, a bounded subset $B^* \subset \ms{X}$ is called absorbing set if for any given bounded set $B$ in the space $\ms{X}$ there is a finite time $T_B \geq 0$ such that $S(t)B \subset B^*$ for all $t  > T_B$. A semiflow is called dissipative if there exists an absorbing set in the state space $\ms{X}$, which also implies the existence of a global attractor $\ms{A}$ if the state space $\ms{X}$ is finite dimensional. Moreover, the global attractor $\ms{A}$ is the $\omega$-limit set of any absorbing set.

Young's inequality is instrumental in differential inequality management. It states that for any positive numbers $x$ and $y$, if $\frac{1}{p} + \frac{1}{q} = 1$ and $p > 1, q > 1$, then
\beq \bl{Yg}
	x\,y \leq \frac{1}{p} \ve x^p + \frac{1}{q} C(\ve, p)\, y^q \leq \ve x^p + C(\ve, p)\, y^q, 
\eeq
where $C(\ve, p) = \ve^{-q/p}$ and the constant $\ve > 0$ can be arbitrarily given.

\section{\textbf{Dissipative Dynamics and Global Attractor}}

We start with the existence and uniqueness of global solutions in time for the initial value problem \eqref{Meq}-\eqref{inc} of Hodgkin-Huxley-Wilson neural network $\mathcal{NW}$. Through the approach of \emph{a priori} estimates we can directly tackle the strong nonlinearity in the model equations and prove the dissipative global dynamics of this neural network system.

\begin{theorem} \label{T1}
	For any given initial state $(V^0, R^0) \in \mathbb{R}^{2n}$, there exists a unique global solution in time,
$$
	V(t; V^0) = \{V_i (t; V_i^0): 1 \leq i \leq n\}, \;\;  R(t; R^0) = \{ R_i(t; R_i^0): 1 \leq i \leq n\}, \;\, t \in [0, \infty),
$$ 
for the initial value problem \eqref{Meq}-\eqref{inc} of Hodgkin-Huxley-Wilson neural network $\mathcal{NW}$.
\end{theorem}

\begin{proof} 
Since the nonlinear functions on the right-hand sides of the two differential equations in \eqref{Meq} are locally Lipschitz continuous, there exists a unique local solution $(V(t), R(t))$ in time to the initial value problem \eqref{Meq}-\eqref{inc}. We want to show that that the maximal existence interval of every local solution is $[0, \infty)$.

Multiply the $V_i$-equation in \eqref{Meq} by $V_i(t)$ for $t \in [0, T_{max})$, which is the maximal existence interval of that local solution and $T_{max}$ may depend on the initial state. Then sum them altogether for all $1 \leq i \leq n$. By \eqref{mR} and note that
$$
	- P \sum_{i=1}^n \sum_{j=1}^n \left[(V_j - V_i)V_i + (V_i - V_j)V_j \right] = \sum_{1 \leq i < j \leq n} P (V_i - V_j)^2 = \frac{1}{2}\, P \sum_{i=1}^n \sum_{j=1}^n (V_i - V_i)^2.
$$
we get	

\begin{equation} \bl{Vi}
	\begin{split}
	&\frac{1}{2} \frac{d}{dt} \sum_{i = 1}^n V_i^2 (t) \leq \frac{1}{2} \frac{d}{dt} \sum_{i = 1}^n V_i^2 (t) + \frac{1}{2} P \sum_{i=1}^n \sum_{j=1}^n (V_i - V_i)^2   \\
	= &\, \frac{1}{2} \frac{d}{dt} \sum_{i = 1}^n V_i^2 (t) - P \sum_{i=1}^n \sum_{j=1}^n \left[(V_j - V_i)V_i + (V_i - V_j)V_j \right]     \\
	= &\, \sum_{i=1}^n \left[ - m_\infty (V_i)(V_i - E_{N_a})V_i - g_K R_i (V_i - E_K)V_i  +  J V_i(t) \right]  \\
	= &\, \sum_{i=1}^n \left[- (a_0 + a_1 V_i + a_2 V_i^2) (V_i - E_{N_a})V_i - g_K R_i (V_i - E_K)V_i  +  J V_i \right]  \\
	= &\, - \sum_{i=1}^n \left[ a_0V_i^2 + a_1 V_i^3 + a_2 V_i^4 + g_K R_i V_i^2  \right]  \\
        &\, + \sum_{i=1}^n \left[E_{N_a} (a_0 V_i + a_1V_i^2 + a_2 V_i^3) + g_K R_i E_K V_i +  J V_i \right], \;\;  t \in [0, T_{max}),
	\end{split}
\end{equation} 
Then multiply the $R_i$-equation in \eqref{Meq} with $ R_i (t)$. Since  $0 < R_\infty (V_i) < H$ according to \eqref{mR}, we have
$$
	\frac{1}{2} \frac{d}{dt} R_i^2 (t) = \frac{1}{\tau_K} (- R_i^2 + R_\infty (V_i) R_i ) \leq \frac{1}{\tau_K} (- R_i^2 + H |R_i |) \leq \frac{1}{2 \tau_K} (- R_i^2(t) + H^2) 
$$
so that
\beq \bl{Ri}
	 \frac{d}{dt} R_i^2 (t) \leq \frac{1}{\tau_K}  (- R_i^2(t) + H^2), \quad t \in [0, T_{max}).
\eeq
By the Gronwall differential inequality, \eqref{Ri} implies that
\beq \bl{Rib}
	R_i^2(t) \leq R_i^2(0)\, e^{-t/\tau_K} + H^2 \leq R_i^2(0) + H^2,  \quad t \in [0, T_{max}).
\eeq  

In order to treat the differential inequality \eqref{Vi} term by term, we can appropriately use the Young's inequality \eqref{Yg} to obtain
\beq \bl{tm}
	\begin{split}
	& (E_{N_a} a_0 + g_K R_i E_K + J_i)V_i \leq \frac{1}{2} a_0 V_i^2 + \frac{1}{2a_0} (E_{N_a} a_0 + g_K R_i E_K + J)^2,   \\[2pt]
	& (E_{N_a} a_1 + g_K R_i )V_i^2 \leq \frac{1}{2} a_2 V_i^4 + \frac{1}{2 a_2} (E_{N_a} a_1 + g_K R_i)^2,   \\[2pt]
	& (a_1 + E_{N_a} a_2)V_i^3 \leq \frac{1}{2} a_2 V_i^4 + \frac{(a_1 +E_{N_a} a_2)^4}{(a_2/2 )^3}, \quad 1 \leq i \leq n.
	\end{split}
\eeq
Substitute \eqref{tm} in the inequality \eqref{Vi}. We have
\beq \bl{Vib}
	\begin{split}
	&\frac{1}{2} \frac{d}{dt} \sum_{i = 1}^n V_i^2 (t) \leq \frac{1}{2} \frac{d}{dt} \sum_{i = 1}^n V_i^2 (t) + \frac{1}{2} P \sum_{i=1}^n \sum_{j=1}^n (V_i - V_i)^2    \\
	= &\, - \sum_{i=1}^n \left[ a_0V_i^2 + a_1 V_i^3 + a_2 V_i^4 + g_K R_i V_i^2  \right]  \\
        &\, + \sum_{i=1}^n \left[E_{N_a} (a_0 V_i + a_1V_i^2 + a_2 V_i^3) + g_K R_i E_K V_i +  J V_i \right]  \\
        \leq &\, - \frac{1}{2} \sum_{i=1}^n a_0V_i(t)^2 + \sum_{i=1}^n \frac{1}{2a_0} (E_{N_a} a_0 + g_K R_i(t) E_K + J )^2   \\
        &\, + \sum_{i=1}^n \left[ \frac{1}{2 a_2} (E_{N_a} a_1 + g_K R_i(t))^2 + \frac{(a_1 +E_{N_a} a_2)^4}{(a_2/2 )^3}\right], \quad  t \in [0, T_{max}),
	\end{split}
\eeq
after plus and minus cancellations of the $\sum_{i=1}^n a_2V_i^4$ terms. Then by the bounding estimates \eqref{Rib} and the non-negativity of the cross-coupling terms on the right-hand side of the first inequality in \eqref{Vib}, it follows that
\beq \bl{VB}
	\begin{split}
	&\frac{d}{dt} \sum_{i = 1}^n V_i^2 (t) \leq - \sum_{i=1}^n a_0V_i(t)^2 + \sum_{i=1}^n \frac{1}{a_0} \left(|E_{N_a}| a_0 + g_K |R_i(t)| |E_K| + |J | \right)^2   \\
        &+ \sum_{i=1}^n \left[ \frac{1}{a_2} (|E_{N_a} a_1| + g_K |R_i(t)|)^2 + \frac{2(a_1 +E_{N_a} a_2)^4}{(a_2/2 )^3}\right]  \\
        \leq &\, - \sum_{i=1}^n a_0V_i^2 (t) + \sum_{i=1}^n \frac{1}{a_0} \left(|E_{N_a}| a_0 + g_K \sqrt{R_i^2(0) + H^2} \,|E_K| + |J | \right)^2   \\
        &\,+ \sum_{i=1}^n \left[ \frac{1}{a_2} \left(|E_{N_a} a_1| + g_K \sqrt{R_i^2(0) + H^2} \right)^2 + \frac{2(a_1 +E_{N_a} a_2)^4}{(a_2/2 )^3}\right], 
	\end{split}
\eeq
for $t \in [0, T_{max})$. Denote the positive constant on the right-hand side of \eqref{VB} by
\beq \bl{M}
	\begin{split}
	M(R^0) &\,=  \sum_{i=1}^n \frac{1}{a_0} \left(|E_{N_a}| a_0 + g_K \sqrt{R_i^2(0) + H^2} \,|E_K| + |J | \right)^2   \\
        &\,+ \sum_{i=1}^n \left[ \frac{1}{a_2} \left(|E_{N_a} a_1| + g_K \sqrt{R_i^2(0) + H^2} \right)^2 + \frac{2(a_1 +E_{N_a} a_2)^4}{(a_2/2 )^3}\right], 
        \end{split}
\eeq
which depends on the initial data $R^0$ shown in \eqref{inc}. Then the differential inequality \eqref{VB} written as
$$
	\frac{d}{dt} \sum_{i = 1}^n V_i^2 (t) \leq - a_0 \sum_{i=1}^n V_i(t)^2 + M(R^0), \quad t \in [0, T_{max}),
$$
shows that any local solution of the initial value problem \eqref{Meq}-\eqref{inc} admits a finite bound which depends on the initial data $(V^0, R^0)$ but is uniform in time $0 \leq t < T_{max}$, namely,
\beq \bl{Sb}
	\sum_{i=1}^n \left(V_i^2(t) + R_i^2(t) \right) \leq \sum_{i=1}^n \left[(V_i^0)^2 e^{- a_0 t} + (R_i^0)^2 e^{- t/\tau_K} \right] + \frac{1}{a_0} M(R^0) + nH^2.
\eeq
It means that all the solutions of the initial value problem \eqref{Meq}-\eqref{inc} will not blow up at any finite time. Therefore, for any initial state $(V^0, R^0) \in \mathbb{R}^{2n}$, there exists a unique global solution in time, the maximal existence interval is $[0, T_{max}) = [0, \infty)$. The proof is completed. 
\end{proof}

Based on the global existence of solutions shown in Theorem \ref{T1}, the solution semiflow $\{S(t): (V^0, R^0) \longmapsto (V(t; V^0), R(t; R^0)\}_{t \geq 0}$ on the state space $\mathbb{R}^{2n}$ is called the Hodgkin-Huxley-Wilson (HHW) neural network semiflow.   

\begin{theorem} \label{T2}
	There exists a bounded absorbing set for the Hodgkin-Huxley-Wilson neural network semiflow $\{S(t)\}_{t \geq 0}$ in the state space $\mathbb{R}^{2n}$, which is the bounded ball 
\beq \label{Bs}
	B^* = \{ (v, r) \in \mathbb{R}^{2n}: \| (v, r) \|^2 \leq G\}.
\eeq 
Here the uniform constant 
\beq \bl{K}
	\begin{split}
	G &\, = 1 + nH^2 + \sum_{i=1}^n \frac{1}{a_0^2} \left(|E_{N_a}| a_0 + g_K \sqrt{1 + H^2} \,|E_K| + |J | \right)^2   \\
        &\,+ \sum_{i=1}^n \left[ \frac{1}{a_0 a_2} \left(|E_{N_a} a_1| + g_K \sqrt{1 + H^2} \right)^2 + \frac{2(a_1 +E_{N_a} a_2)^4}{a_0 (a_2/2 )^3}\right]
        \end{split}
\eeq
is independent of any initial state $(V^0, R^0) \in \mathbb{R}^{2n}$. $\{S(t)\}_{t \geq 0}$ is a dissipative dynamical system and there exists a global attractor $\ms{A} \subset \mathbb{R}^{2n}$ for this semiflow.
\end{theorem}

\begin{proof}
From \eqref{M} and the estimate \eqref{Sb} of global solutions for this Hodgkin-Huxley-Wilson neural network $\mathcal{NW}$ and more precisely with $R_i^2(0)$ replaced by $R_i^2(0) e^{-t/\tau_K}$ therein due to \eqref{Rib}, since $e^{- a_0 t} \to 0$ and $e^{-t/\tau_K}$ as $t \to \infty$, we can assert that
\beq \label{lsp}
	\limsup_{t \to \infty} \left\|\left(V(t), R(t)\right) \right\|^2 = \limsup_{t \to \infty} \sum_{i=1}^n \left(V_i^2(t) + R_i^2(t) \right) < G
\eeq
for all the solutions of \eqref{Meq} with any initial state $(V^0, R^0) \in \mathbb{R}^{2n}$. Moreover, \eqref{Sb} also shows that for any given bounded set $B = \{(v, r) \in \mathbb{R}^{2n}: \|(v, r)\|^2 \leq L\}$ in the state space, there exists a finite time
\beq \bl{TB}
	T_B \geq \max \left\{\frac{1}{a_0}, \, \tau_K \right\} \log^+ (L/G) \geq 0
\eeq
such that all the solution trajectories started at the initial time $t = 0$ from inside that set $B$ will permanently enter the bounded ball $B^*$ shown in \eqref{Bs} when $t > T_B$.  Therefore, the bounded ball $B^*$ is an absorbing set for the semiflow $\{S(t)\}_{t \geq 0}$ so that this Hodgkin-Huxley-Wilson neural network semiflow is a dissipative dynamical system..

Finally, since the state space $\mathbb{R}^{2n}$ is a finite-dimensional Hilbert space, the closed bounded absorbing set $B^*$ is a compact set. According to \cite[Theorem 1.1]{CV}, there exists a global attractor 
$$
	\mathscr{A} = \gw (B^*) = \bigcap_{\tau \geq 0}\; \overline{\bigcup_{t \geq \tau} S(t) B^*}
$$
which is the $\omega$-limit set of the absorbing set $B^*$ for this semiflow.
\end{proof}

\section{\textbf{Complete Exponential Synchronization of HHW Neural Networks}}  

In this section we shall prove the main result on complete synchronization of the Hodgkin-Huxley-Wilson neural networks described by the proposed model \eqref{Meq}. 
For this HHW neural network $\mathcal{NW}$ with the assumptions \eqref{Asp}, we define the interneuron gap function between any two neurons $\mathcal{N}_i$ and $\mathcal{N}_j$ to be 
\begin{gather*}
	U_{ij} (t) = V_i(t) - V_j (t),  \quad   \Pi_{ij} = R_i(t) - R_j(t), \quad  \text{for} \;\;  1\leq i, j \leq n. 
\end{gather*}
The differencing equations for $U_{ij}(t)$ and $\Pi_{ij}(t)$ are
\beq \bl{Deq} 
	\begin{split}
	\frac{dU_{ij}}{d t} = &\,- (m_\infty (V_i) - m_\infty (V_j) ) (V_i - E_{N_a}) - m_\infty (V_j) U_{ij}   \\
	&\, - g_K (R_i(V_i) - R_j(V_j)) (V_i - E_K) - g_K R_j(V_j) U_{ij} - n PU_{ij},   \\
	\frac{d\Pi_{ij}}{d t} = &\, \frac{1}{\tau_K} \left(-\Pi_{ij} + R_\infty (V_i) - R_\infty (V_j) \right),  \quad \text{for} \;\; t > 0, \;\; 1 \leq i, j \leq n,
	\end{split}
\eeq
where the coupling term $\sum_{\ell=1}^n P[(V_\ell - V_i) - (V_\ell - V_j)] = - P \sum_{\ell=1}^n (V_i -V_j) = - n PU_{ij}$. Now we present and prove the main result of this paper.  
\begin{theorem} \bl{ThM}    
	The Hodgkin-Huxley-Wilson neural network $\mathcal{NW}$ described by the model \eqref{Meq}-\eqref{mR} with the assumptions \eqref{Asp} is exponentially synchronizable. If the threshold condition
\beq \bl{Resh}
	P > P^*= \max \left\{0, \; \frac{1}{n} \left[Q - \left(a_0 + \frac{1}{2\tau_K} \gl^2 H^2\right)\right] \right\}
\eeq 
is satisfied by the network coupling strength, where the positive constant
\beq \bl{Q}
	Q = a_1 E_{N_a} + g_K (1 + 2H) ( |E_K| + 1) + 6 \left[\frac{a_1^2}{a_2} + \frac{ a_2}{E_{N_a}^2} + \frac{1}{a_2} g_K^2 (1 + 2H)^2\right] 
\eeq 
is independent of initial states, then the Hodgkin-Huxley-Wilson neural network $\mathcal{NW}$ is completely synchronized at an exponential convergence rate
\beq \bl{rate}
	\mu (P) = \min \left\{ \frac{1}{2\tau_K}, \left[ a_0 + \frac{1}{2\tau_K} \gl^2 H^2 + n P - Q \right] \right\} > 0.
\eeq
\end{theorem}

\begin{proof}   
We denote $U_{ij}(t)$ by $U(t)$ and $\Pi_{ij}(t)$ by $\Pi(t)$ for any given indices $(i, j), 1 \leq i \neq j \leq n$. The proof will go through three steps. 

Step 1. Multiply the $\Pi_{ij}$-equation in \eqref{Deq} by $\Pi_{ij}(t), 1 \leq i \neq j \leq n$. We get
\beq \bl{Pi}
	\begin{split}
	&\frac{1}{2} \frac{d}{dt}\Pi^2(t) = \frac{1}{\tau_K} \left[- \Pi^2(t) + \Pi (t) (R_\infty (V_i) - R_\infty (V_j)) \right]   \\
	= &\, \frac{1}{\tau_K} \left[- \Pi^2(t) + \Pi (t) (V_i(t) - V_j(t)) \frac{\gl H e^{- \gl(\xi - E_K)}}{(1 + e^{- \gl (\xi - E_K)})^2} \right]   \\
	\leq &\, \frac{1}{\tau_K} \left[- \Pi^2(t) + \gl H\, \Pi (t) U(t)\right], \quad t > 0, 
	\end{split}
\eeq
where the mean value theorem is used and $\xi$ is between $V_i(t)$ and $V_j(t)$. Then multiply the $U_{ij}$-equation in \eqref{Deq} by $U_{ij}(t)$ to get
\beq \bl{Ut}   
	\begin{split}
	\frac{1}{2} &\,\frac{d}{dt}\,U^2 = - (m_\infty (V_i) - m_\infty (V_j) ) (V_i - E_{N_a}) U - m_\infty (V_j) U^2  \\[6pt]
	&\, - g_K (R_i(t, V_i) - R_j(t, V_j)) (V_i - E_K) U - g_K R_j (t, V_j) U^2 - n PU^2   \\[6pt]
	= &\, -(a_1U + a_2 U (V_i + V_j) )(V_i - E_{N_a})U - (a_0 + a_1V_j + a_2 V_j^2)U^2   \\[6pt]
	&\, - g_K \Pi (V_i - E_K)U - g_K R_j(t, V_j) U^2 - n PU^2   \\[4pt]
	\leq &\, - (a_0 + a_2 (V_i^2 + V_j^2))U^2 + \left[ a_1V_j + a_2 E_{N_a}V_i + (E_{N_a} - V_i) (a_1 + a_2 V_j) \right] U^2   \\[6pt]
	&\, + g_K \Pi (E_K - V_i) U - g_K R_j(t, V_j) U^2 - n PU^2,  \quad  t > 0,
	\end{split}
\eeq 
in which the signs of constants $g_K, E_{N_a}, E_K$ are used. Below we can simply write $R_j(t, V_j) = R_j(t)$. 

Step 2. Now we take the approach of grouping estimates to discover the key differencing inequality of the gap functions among the neurons of the entire neural network $\mathcal{NW}$. Add up the inequalities \eqref{Pi} and \eqref{Ut}. Since $|\Pi (t)| = |R_i(t) - R_j(t)| \leq \|R^0\| e^{-t/(2\tau_K)} + 2H \to 2H$ as $t \to \infty$. So there exists a finite iime $T_0(R^0) > 0$ such that $|\Pi(t)| < 1 + 2H$ for $t > T_0(R^0)$. Thus we obtain

\beq \bl{PU}
	\begin{split}
	\frac{1}{2}&\, \frac{d}{dt} (U^2 + \Pi^2) \leq \frac{1}{\tau_K} \left(- \Pi^2 + \gl H\, \Pi \,U\right)     \\[5pt]
	&\, - (a_0 + a_2 (V_i^2 + V_j^2))U^2 + \left[ a_1V_j + a_2 E_{N_a}V_i + (E_{N_a} - V_i) (a_1 + a_2 V_j) \right] U^2    \\[7pt]
	&\, + g_K \Pi (E_K - V_i) U - g_K R_j(t) U^2 - n PU^2    \\[5pt] 
	= &\, \frac{1}{\tau_K} \left(- \Pi^2 + ( \gl H + \tau_K g_K E_K) \Pi \,U \right)     \\[5pt]
	&\, - (a_0 + a_2 (V_i^2 + V_i V_j + V_j^2))U^2 + \left[ a_1V_j + a_2 E_{N_a}(V_i + V_j) + a_1 (E_{N_a} - V_i) \right] U^2  \\[7pt]
	&\, - g_K V_i \,\Pi \,U - g_K R_j(t) U^2 - n PU^2    \\[3pt]
	\leq &\, - \left[ \frac{1}{\tau_k} \Pi^2 + \left(a_0 + a_2 (V_i^2 + V_i V_j + V_j^2 \right) + n P)U^2 \right]   \\
	&\, + \frac{1}{\tau_K} \left[ \gl H + \tau_K g_K E_K (1 + 2H)\right] \Pi \,U  \\[5pt]
	&\, + [a_1 (E_{N_a}- V_i + V_j) + a_2 E_{N_a}(V_i + V_j)]U^2 + g_K (1 + 2H) |V_i \,U| - g_K R_j(t) U^2.
	\end{split}
\eeq
To handle the above differential inequality \eqref{PU}, by Cauchy inequality and \eqref{Rib} we can deduce that 
\beq \bl{gT}
	\begin{split}
	&V_i^2 + V_i V_j + V_j^2 \geq \frac{1}{2}\, (V_i^2 + V_j^2),   \\
	&\frac{1}{\tau_K} \left[ \gl H + \tau_K g_K E_K (1 + 2H)\right] U \leq \frac{1}{2\tau_K} \Pi^2 + \frac{1}{2\tau_K} \left[ \gl H + \tau_K g_K E_K (1 + 2H)\right]^2 U^2,   \\
	&a_1 (-V_i + V_j) \leq \frac{6a_1^2}{a_2} + \frac{1}{6}\, a_2\, (V_i^2 + V_j^2),   \\
	&a_2 E_{N_a}(V_i + V_j) \leq \frac{6 a_2}{E_{N_a}^2} + \frac{1}{6} a_2 (V_i^2 + V_j^2),  \\
	&\, g_K (1 + 2H) |V_i \,U| \leq \frac{6}{a_2} g_K^2 (1 + 2H)^2 U^2 + \frac{1}{6}\, a_2\, V_i^2,    \\
	- &\, g_K R_j(t) \leq g_K \sqrt{R_i^2(0) e^{- t/ \tau_K} + H^2} \leq g_K ( |R_i(0)| e^{- t/(2\tau_K)} + H) .
	\end{split}
\eeq 
Substitute \eqref{gT} in the last step of the inequality \eqref{PU}. Then we reach the following Gronwall-type differential inequality

\beq \bl{Gw}
	\begin{split}
	&\frac{1}{2} \frac{d}{dt} (U^2(t) + \Pi^2(t))    \\[2pt]
	\leq &\, - \left[ \frac{1}{\tau_k} \Pi^2 + \left(a_0 + a_2 (V_i^2 + V_i V_j + V_j^2 \right) - \left(a_1 E_{N_a} + g_K  E_K \Pi \, U \right) + n P)U^2 \right]   \\ 
	&\, + \frac{1}{\tau_K} \left[ \gl H + \tau_K g_K E_K (1 + 2H)\right] U + \left[ a_1 (- V_i + V_j) + a_2 E_{N_a}(V_i + V_j) \right] U^2   \\[3pt]
	&\, + g_K (1 + 2H) |V_i \,U | - g_K R_j(V_j) U^2   \\[4pt] 
	\leq &\, - \left[ \frac{1}{\tau_k} \Pi^2 + \left(a_0 + \frac{1}{2} a_2 (V_i^2  + V_j^2) \right) U^2 - \left(a_1 E_{N_a} + g_K (1 + 2H) |E_K| \right) + n P)U^2 \right]   \\
	&\, + \frac{1}{2\tau_K} \Pi^2 + \frac{1}{2\tau_K} \left[ \gl H + \tau_K g_K E_K (1 + 2H)\right]^2 U^2   \\
	&\, + \left[\frac{1}{6}\, a_2\, (V_i^2 + V_j^2) + \frac{1}{6} a_2 (V_i^2 + V_j^2) + \frac{1}{6}\, a_2\, V_i^2 \right] U^2  + \frac{6}{a_2} g_K^2 (1 + 2H)^2 U^2. \\
	&\, + g_K ( |R_i(0)| e^{- t/(2\tau_K)} + H) U^2 + 6 \left[\frac{a_1^2}{a_2} + \frac{ a_2}{E_{N_a}^2} + \frac{1}{a_2} (g_K \gl H)^2\right] U^2   \\
	\leq &\, - \left[ \frac{1}{2\tau_k} \Pi^2 + \left(a_0 + \frac{1}{2\tau_K} \gl^2 H^2 + n P \right)U^2 \right] + g_K (1 + H) U^2   \\
	&\, + \left(a_1 E_{N_a} + g_K (1 + 2H) |E_K| \right) U^2 + 6 \left[\frac{a_1^2}{a_2} + \frac{ a_2}{E_{N_a}^2} + \frac{1}{a_2} g_K^2 (1 + 2H)^2 \right] U^2, \quad t > 0,
	\end{split}
\eeq
where we see the cancellation of the terms 
$$
	- \frac{1}{2} a_2 (V_i^2 + V_j^2) + \left[\frac{1}{6}\, a_2\, (V_i^2 + V_j^2) + \frac{1}{6} a_2 (V_i^2 + V_j^2) + \frac{1}{6}\, a_2\, V_i^2 \right] \leq 0.
$$
Therefore, for any $t > T_0(R^0)$, it holds that 
\beq \bl{key}
	\begin{split}
	&\frac{1}{2} \frac{d}{dt} (U^2(t) + \Pi^2(t))    \\
	\leq &\, - \left[ \frac{1}{2\tau_k} \Pi^2(t) + \left(a_0 + \frac{1}{2\tau_K} \gl^2 H^2 + n P \right) U^2(t) \right]   \\
	&\, + \left[a_1 E_{N_a} + g_K (1 + 2H) ( |E_K| + 1) + 6 \left[\frac{a_1^2}{a_2} + \frac{ a_2}{E_{N_a}^2} + \frac{1}{a_2} g_K^2 (1 + 2H)^2\right] \right] U^2(t).   
	\end{split}
\eeq
The sum of positive coefficients in the rightmost term of \eqref{key} is denoted by $Q$ shown in \eqref{Q}.
 
Now we have shown that the gap functions $U(t) = U_{ij}(t) = V_i(t) - V_j(t)$ and $\Pi(t) = \Pi_{ij}(t) = R_i(t) -R_j(t)$ for all neuron pairs $(\mathcal{N}_i, \mathcal{N}_j), 1 \leq i \neq j \leq n$, in the Hodgkin-Huxley-Wilson neural network $\mathcal{NW}$ satisfy the Gronwall inequality
\beq \bl{Fnl}
	\frac{1}{2} \frac{d}{dt} (U^2(t) + \Pi^2(t)) \leq - \left[ \frac{1}{2\tau_k} \Pi^2(t) + \left(a_0 + \frac{1}{2\tau_K} \gl^2 H^2 + n P  - Q\right) U^2(t) \right], \;\, t > T_0.
\eeq

Step 3. Finally we can solve this linear differential inequality with respect to all the gap functions $(U^2(t), \Pi^2(t))$ and establish an explicitly computable threshold condition for the network coupling strength $P$ to achieve complete synchronization. 

If the threshold condition \eqref{Resh} in this theorem is satisfied, then 
$$
	a_0 + \frac{1}{2\tau_K} \gl^2 H^2 + n P > Q.
$$
From the inequality \eqref{Fnl} we can confirm
\beq \bl{UR}
	\frac{1}{2} \frac{d}{dt} (U_{ij}^2 + \Pi_{ij}^2) \leq - \min \left\{ \frac{1}{2\tau_k}, \left[a_0 + \frac{1}{2\tau_K} \gl^2 H^2 + n P - Q \right] \right\} (U_{ij}^2 +  \Pi_{ij}^2), \;\; t > T_ 0.
\eeq
Therefore, for any solution $(V(t), R(t))$ of the model equations \eqref{Meq}-\eqref{mR}, all the gap functions $U_{ij}(t), \Pi_{ij}(t)$ as $t \to \infty$ converge to zero at an exponential rate $\mu (P)$ shown in \eqref{rate}. Namely, for any $1 \leq i < j \leq n$,
\beq \bl{cov}
	\begin{split}
	\| U_{ij}(t) &\, \|^2 + \|\Pi_{ij}(t)\|^2 = \|V_i(t) - V_j(t)\|^2 + \|R_i(t) - R_j(t)\|^2   \\[2pt]
	\leq &\, e^{- \mu(P) (t - T_0)}\, (\|V_i (T_0) - V_j (T_0)\|^2 + \|R_i^(T_))- R_j^(T_0)|^2) \to 0, \;\; \text{as} \;\; t \to \infty.
	\end{split}
\eeq
According to Definition \eqref{Def}, the Hodgkin-Huxley-Wilson neural network $\mathcal{NW}$ is complete exponentially synchronized if the threshold condition is satisfied. The proof is completed. 
\end{proof}

The homogeneous setting of the Hodgkin-Huxley-Wilson neural networks described by the model \eqref{Meq}-\eqref{mR} assumes all the neurons in the network have the same biological parameters. Realistically neurons may exhibit parameter mismatch and form a heterogeneous network. For any heterogeneous Hodgkin-Huxley-Wilson neural network, one can aim to prove approximate synchronization defined and treated in the author's recent publication \cite{YC}.

\section{\textbf{Fractional Memristive Hodgkin-Huxley-Wilson Neural Networks}}  

In recent two decades, research on time-fractional differential equations has been very active and extensive in modeling dynamics of interdisciplinary processes \cite{BJ} ranging from viscoelastic liquids, image processing, encryption, dielectric polarization, to mathematical biology. Time-fractional differential equation models featuring fractional-order time derivatives have the advantage of continuum long-term memory properties in comparison to classical integer-order ODE with discrete time-delay terms. With the fractional-order time derivatives the neural networks exhibit richer dynamic behavior including intermediate regimes between stability and chaos, enhanced adaptability to time-varying inputs, and finer control of convergence rates. It is also recognized \cite{XY} that memristor is a promising element in modeling the spiking neural networks to mimic memorizing and learning functions of human brain.

In this section we shall extend the synchronization Theorem \ref{ThM} to another new model called fractional memristive Hodgkin-Huxley-Wilson neural networks composed of $n\, (\geq 2)$ cross-coupled neurons, denoted by $\mathcal{MW}$. Each neuron $\mathcal{N}_i, 1 \leq i \leq n$, is presented by the Caputo fractional differential equations: 
\beq \bl{Feq}
	\begin{split}
	D_c^\ap V_i (t) &\, = - m_\infty (V_i)(V_i - E_{N_a}) - g_K R_i(V_i) (V_i - E_K)   \\[3pt]
	&\, + k\, \psi(\rho) V_i + J + \sum_{j=1}^n P(V_j - V_i),  \;\;  1 \leq i \leq n,  \\
	D_c^\ap R_i (t) &\, =  \frac{1}{\tau_K} \left(-R_i + R_\infty (V_i) \right),   \;\; 1 \leq i \leq n,   \\
	D_c^\ap \rho(t) &\, = \sum_{i=1}^n \ga_i V_i - b \rho,  \quad  \text{for}  \;\;  t > 0, \;\; \ap \in (0, 1).
	\end{split} 
\eeq
Here the Caputo fractional time-derivative of order $\ap \in (0, 1)$ denoted by $\dca$ of a scalar function $y(t)$ is defined by 
$$
	\dca y(t) = \frac{1}{\G (1 - \ap)} \int_0^t (t - s)^{-\ap} y^\prime (s)\, ds, \quad  t > 0,
$$
and $\G(z) = \int_0^\infty t^{z-1} e^{-t} dt, \,z > 0,$ is the Gamma function, for a scalar function $y(\cdot) \in C[0, \infty)$ and $y^\prime (\cdot) \in L^1_{loc} [0, \infty)$.  The initial states of the system \eqref{Feq} will be denoted by 
\beq \bl{ic}
	 V_i^0 = V_i(0),  \quad R_i^0 = R_i(0),  \quad \rho^0 = \rho (0), \;\; 1 \leq i  \leq n.
\eeq
Beside the assumption \eqref{Asp} still in place, we assume that the memristive coupling constant $k > 0$ and the memristor window function $\psi (s) = s(1 - \gb s),\, s \in \mathbb{R}$. The constants $\gb > 0, \ga_i  \in \mathbb{R}$ and $b > 0$. The basic theory and some applications of fractional calculus and fractional differential eqations can be found in books \cite{KD, BJ} and references therein.

In order to pursue the asymptotic dynamics for solutions of this new fractional memristive model, we recall the following lemmas presented and briefly proved in \cite{YY}. 
\begin{lemma} \bl{L1}
       For $\ap \in (0, 1)$ and any given  $T > 0$, if $f \in AC [0, T]$, then it holds that
\beq \bl{prl}
	 \frac{1}{2}\, \dca f^2(t) \leq f(t) \dca f(t), \;\; \text{for} \;\; t \in (0,T).
\eeq
\end{lemma}

\begin{lemma} \textup{(Fractional Gronwall Inequality)}  \bl{L2}
      If a nonnegative function $x(t)$ is absolutely continuous for $t \geq 0$ and satisfies the inequality
\beq \bl{pq}
	\dca x(t) \leq p - q\, x(t), \;\;  \text{for} \;\;  t > 0,
\eeq
where the fractional order $\ap \in (0, 1)$, $p$ and $q$ are positive constants, then it holds that
\beq \bl{fG}
	\begin{split}
	x(t) &\,\leq x(0)\, E_\ap (-q\, t^\ap) + \, \frac{p}{q}\,\G(\ap) \left(1 - E_\ap (- q\, t^\ap) \right)      \\
	&\, < x(0) \left( \frac{\G(1 + \ap)}{\G(1 + \ap) + q t^\ap} \right) + \, \frac{p}{q}\,\G(\ap),  \;\;  \text{for} \; \; t > 0.
	\end{split}
\eeq
\end{lemma}
\noindent The detailed proof of this Lemma can be seen in \cite[Lemma 2]{YY} with citations therein.

The Mittag-Leffler function $E_\ap (z)$ and the two-parameter Mittag-Leffler function $E_{\ap_1, \ap_2} (z)$ \cite{BJ} defined by  
$$
	E_\ap (z) = \sum_{n = 0}^\infty \frac{z^n}{\G(n \ap + 1)},  \quad  E_{\ap_1, \ap_2} (z) = \sum_{n = 0}^\infty \frac{z^n}{\G(n \ap_1 + \ap_2)},  \quad z \in \mathbb{C},
$$
played key roles in fractional differential equations (FDE) and fractional analysis. $E_\ap (z)$ is an entire function, in particular $E_1(z) = e^z$. The function $E_\ap (-x)$ with real variable $x > 0$ is especially more involved in treatment of FDE. 

\begin{lemma} \bl {L3}
	For an initial value problem of autonomous Caputo-fractional differential equation\textup{(}s\textup{)}
$$
	\dca x(t) = f(x), \quad x(0) = x_0 \in \Omega \subset \mathbb{R}^d,
$$
if the scalar or vector function $f(x)$ is locally Lipschitz on its bounded or unbounded domain $\Omega \subset \mathbb{R}^d$ and a solution of this initial value problem satisfies 
\beq \bl{dp}
	\|x(t)\| \leq \|x_0\|\, C(t) + L, \quad  t \in I_{max} = [0, T_{max}),
\eeq
where $C(t)$ is a positive non-increasing continuous function, $C(t) \to 0$ as $t \to \infty$, and $L > 0$ is a constant, then the maximal existence interval of this solution $I_{max} = [0, \infty)$ so that $x(t)$ is a global solution in time.  
\end{lemma}

It has been proved that autonomous Caputo fractional differential equations of oder $\ap \in (0, 1)$ on a domain $\Omega \subset \mathbb{R}^d$ satisfying the global Lipschitz condition (formulated into weakly singular integral equations) generate semi-dynamical systems in a space of time-continuous functions \cite{DK}. For autonomous Caputo fractional differential equations of oder $\ap \in (0, 1)$ only satisfying local Lipschitz condition, it is known that under certain dissipative conditions there exist absorbing sets in the state space $\mathbb{R}^d$ for solutions  \cite{PK}.

\begin{theorem} \label{T3}
	For the fractional memristive Hodgkin-Huxley-Wilson neural network $\mathcal{MW}$ presented by the model \eqref{Feq} with the assumption \eqref{Asp} and $k > 0, \gb > 0, a_0 > k/\gb$, there exists a bounded absorbing set in the state space $\mathbb{R}^{2n}$, which is the ball
\beq \label{Br}
	B_{\ap}^* = \{ (v, r) \in \mathbb{R}^{2n}: \| (v, r) \|^2 \leq G_{\ap}\}.
\eeq 
where $G_{\ap} > 0$ given in \eqref{Ga} independent of any initial state is a uniform ultimate bound for all solutions.
\end{theorem}

\begin{proof}
	The memductance-potential synapse function in the first equation of this neural network model \eqref{Feq} has the property
\beq \bl{me}
	k \psi (\rho) V_i^2 = k \rho(1 - \gb \rho) V_i^2 = k (\rho - \gb \rho^2) V_i^2 \leq k \left(\frac{1}{2\gb} - \frac{\gb}{2} \rho^2\right)V_i^2 \leq \frac{k}{2\gb} V_i^2,
\eeq
which is plugged in the estimate inequality similar to \eqref{Vi} by Lemma \ref{L1}. Then we get
\begin{equation} \bl{mVi}
	\begin{split}
	&\frac{1}{2} \dca \left(\sum_{i = 1}^n V_i^2 (t)\right) \leq \frac{1}{2} \dca \left(\sum_{i = 1}^n V_i^2 (t) \right) + \frac{1}{2} P \sum_{i=1}^n \sum_{j=1}^n (V_i - V_i)^2    \\
	= &\, - \sum_{i=1}^n \left[ \left((a_0 - \frac{k}{2\gb}\right) V_i^2 + a_1 V_i^3 + a_2 V_i^4 + g_K R_i V_i^2 \right]     \\
        &\, + \sum_{i=1}^n \left[E_{N_a} (a_0 V_i + a_1V_i^2 + a_2 V_i^3) + g_K R_i E_K V_i +  J V_i \right], \;  t \in [0, T_{max}).
	\end{split}
\end{equation} 
Follow the steps in the proof of Theorem \ref{T1}, consequently it holds that
\beq \bl{DV}
	\dca \sum_{i = 1}^n V_i^2 (t) \leq - \left(a_0 - \frac{k}{\gb}\right) \sum_{i=1}^n V_i(t)^2 + M^*(R^0), \quad t \in [0, T_{max}),
\eeq
where 
\beq \bl{Ma}
	\begin{split}
	M^*(R^0) &\,=  \sum_{i=1}^n \frac{1}{(a_0 - \frac{k}{\gb})} \left[ |E_{N_a}| \left(a_0 - \frac{k}{\gb}\right) + g_K \sqrt{R_i^2(0) + H^2} \,|E_K| + |J | \right]^2   \\
        &\,+ \sum_{i=1}^n \left[ \frac{1}{a_2} \left(|E_{N_a} a_1| + g_K \sqrt{R_i^2(0) + H^2} \right)^2 + \frac{2(a_1 +E_{N_a} a_2)^4}{(a_2/2 )^3}\right].
        \end{split}
\eeq
By the fractional Gronwall inequality \eqref{fG} stated in Lemma \ref{L2}, \eqref{DV} implies that
\beq \bl{DVb}
         \sum_{i = 1}^n V_i^2 (t) <   \sum_{i = 1}^n V_i^2(0) \left( \frac{\G(1 + \ap)}{\G(1 + \ap) + (a_0 - k/\gb) t^\ap} \right) + \, \frac{M^*(R^0)}{(a_0 - k/\gb)}\,\G(\ap).
\eeq
Note that the fractional order $\ap \in (0, 1)$ so that 
$$
	\lim_{t \to \infty} \frac{\G(1 + \ap)}{\G(1 + \ap) + (a_0 - k/\gb) t^\ap} = 0.
$$
Therefore all the solutions of \eqref{Feq} exists for $t \in [0, \infty)$, because they will not blow up at any finite time. Moreover, similar to \eqref{K} and \eqref{lsp}, here it follows that
$$
	\limsup_{t \to \infty} \left\|\left(V(t), R(t)\right) \right\|^2 = \limsup_{t \to \infty} \sum_{i=1}^n \left(V_i^2(t) + R_i^2(t) \right) < G_{\ap}
$$
and 
\beq \bl{Ga}
	\begin{split}
	G_{\ap} = &\, 1 + nH^2 + \sum_{i=1}^n \frac{\G(\ap)}{(a_0 - \frac{k}{\gb})^2} \left[|E_{N_a}| \left(a_0 - \frac{k}{\gb} \right) + g_K \sqrt{1 + H^2} \,|E_K| + |J | \right]^2   \\
        &\,+ \sum_{i=1}^n \frac{\G(\ap)}{(a_0 - \frac{k}{\gb})} \left[\frac{1}{ a_2} \left(|E_{N_a} a_1| + g_K \sqrt{1 + H^2} \right)^2 + \frac{2(a_1 +E_{N_a} a_2)^4}{(a_2/2 )^3}\right].
        \end{split}
\eeq
It shows that $B^*_{\ap}$ in \eqref{Br} is an absorbing set for the solutions of this fractional memristive Hodgkin-Huxley-Wilson neural network $\mathcal{MW}$ and its asymptotic dynamics is globally dissipative. The positive constant $G_\ap$ is a uniform ultimate bound for all solutions in the state space $\mathbb{R}^{2n}$.
\end{proof}

As a corollary, it can be shown that the memductance function $\rho(t)$ as an auxiliary state variable of this neural network also admits a predicable ultimate upper bound. From the $\rho$-equation in \eqref{Feq} and the ultimate upper bound $G_\ap$ of $\sum_{i=1}^n V^2_i (t)$ in Theorem \ref{T3}, we have   
\beq \bl{Dro}
	\frac{1}{2} \dca \rho^2(t) \leq \sum_{i=1}^n\, \ga_i V_i(t) \rho(t) - b\, \rho^2 (t) \leq \frac{1}{2b} \left[\sum_{i=1}^n\, \ga_i V_i(t)\right]^2 - \frac{b}{2}\, \rho^2(t).
\eeq
By Theorem \ref{T3}, there is a finite time $T_{\ap}(V^0, R^0) \geq 0$ such that $\sum_{i=1}^n V^2_i(t) \leq G_{\ap}$ for all $t > T_{\ap}$. It follows from \eqref{Dro} that
$$
	 \dca \rho^2(t) \leq \frac{G_{\ap}}{b} \max \{\ga_i^2: 1 \leq i \leq n\} - b\, \rho^2 (t),  \quad  T > T_{\ap}.
$$
Then we can use Lemma \ref{L2} to obtain
$$
	\rho^2(t)\leq \rho^2(T_{\ap}) \left( \frac{\G(1 + \ap)}{\G(1 + \ap) + b\, t^\ap} \right) + \frac{G_{\ap}}{b^2} \max \{\ga_i^2: 1 \leq i \leq n\} \,\G(\ap),  \quad   t > T_{\ap}.
$$
Consequently, we get an ultimate upper bound of the memductance function:
\beq 
	\limsup_{t \to \infty} \rho^2(t) < 1 + \frac{G_{\ap}}{b^2} \max \{\ga_i^2: 1 \leq i \leq n\} \,\G(\ap).
\eeq 

Finally we prove that the fractional memristive Hodgkin-Huxley-Wilson neural network $\mathcal{MW}$ is complete synchronizable at a fractional-power convergence rate. 

\begin{theorem} \bl{T4}    
	Under the assumption \eqref{Asp} and $k > 0, \gb > 0, a_0 > k/\gb$, Caputo fractional memristive Hodgkin-Huxley-Wilson neural network $\mathcal{MW}$ described by the model \eqref{Feq} is complete synchronizable. If the threshold condition
\beq \bl{Pnew}
	P > P_* = \max \left\{0, \; \frac{1}{n} \left[Q + \frac{k}{2\gb} - \left(a_0 + \frac{1}{2\tau_K} \gl^2 H^2\right)\right] \right\}
\eeq 
is satisfied by the network coupling strength, where the positive constant $Q$ in \eqref{Q} is independent of initial states, then the neural network $\mathcal{MW}$ is complete synchronized at an asymptotic fracttional-power convergence rate
\beq  \bl{map}
	\mu_{\ap} (P, t) = \frac{\G(1 + \ap)}{\G (1 +\ap) + 2\delta (P)\, t^{\ap}},
\eeq
where
\beq \bl{rt}
	\delta (P) = \min \left\{ \frac{1}{2\tau_K}, \left[ a_0 + \frac{1}{2\tau_K} \gl^2 H^2 + n P - Q - \frac{k}{2\gb} \right] \right\} > 0.
\eeq  
\end{theorem}

\begin{proof}   
For any given indices $1 \leq i \neq j \leq n$, denote by $W_{ij}(t) = V_i(t) - V_j(t)$ and $\Pi_{ij}(t) = R_i(t) - R_j(t)$ in regard to this fractional memristive Hodgkin-Huxley-Wilson neural network $\mathcal{MW}$.  

By the memristive synapse property \eqref{me}, similar to \eqref{Ut} we have the estimation inequality about the gap functions of the membrane potential:
\beq \bl{Wt}   
	\begin{split}
	\frac{1}{2} \dca\,W_{ij}^2 \leq &\, - \left(a_0 - \frac{k}{2\gb} + a_2 (V_i^2 + V_j^2)\right) W_{ij}^2   \\
	&\, + \left[ a_1V_j + a_2 E_{N_a}V_i + (E_{N_a} - V_i) (a_1 + a_2 V_j) \right] W_{ij}^2   \\[6pt]
	&\, + g_K \Pi_{ij} (E_K - V_i) W_{ij} - g_K R_j(t, V_j) W_{ij}^2 - n PW_{ij}^2,  \quad  t > 0,  
	\end{split}
\eeq 
and similar to \eqref{PU} the combined estimation inequality together with the gap functions of the recovery potassium channel: 
\beq \bl{PW}
	\begin{split}
	&\frac{1}{2} \dca (W_{ij}^2 + \Pi_{ij}^2) \leq - \left[ \frac{1}{\tau_k} \Pi_{ij}^2 + \left( a_0 - \frac{k}{2\gb} + a_2 (V_i^2 + V_i V_j + V_j^2) + n P\right)W_{ij}^2 \right]   \\ 
	&\, + \frac{1}{\tau_K} \left[ \gl H + \tau_K g_K E_K (1 + 2H)\right] \Pi_{ij} \,W_{ij}  \\[5pt]
	&\, + [a_1 (E_{N_a}- V_i + V_j) + a_2 E_{N_a}(V_i + V_j)]W_{ij}^2 + g_K (1 + 2H) |V_i \,W_{ij}| - g_K R_j(t) W_{ij}^2.
	\end{split}
\eeq 
Through the same term by term  treatment in \eqref{PW} as in \eqref{gT}, we obtain the following Gronwall-type differential inequality. There exists a finite time $T^0 >0$ depending on the initial state $(V^0, R^0)$ such that 
\beq \bl{ky} 
	\begin{split}
	&\, \frac{1}{2} \dca (W_{ij}^2(t) + \Pi_{ij}^2(t))    \\
	\leq &\, - \left[ \frac{1}{2\tau_k} \Pi_{ij}^2(t) + \left(a_0 - \frac{k}{2\gb} + \frac{1}{2\tau_K} \gl^2 H^2 + n P \right) W_{ij}^2(t) \right]   \\
	&\, + \left[a_1 E_{N_a} + g_K (1 + 2H) ( |E_K| + 1) + 6 \left[\frac{a_1^2}{a_2} + \frac{ a_2}{E_{N_a}^2} + \frac{1}{a_2} g_K^2 (1 + 2H)^2\right] \right] W_{ij}^2(t). \\
	= &\, \left[ \frac{1}{2\tau_k} \Pi_{ij}^2(t) + \left(a_0 - \frac{k}{2\gb} + \frac{1}{2\tau_K} \gl^2 H^2 + n P - Q \right) W_{ij}^2(t) \right]  
	\end{split}
\eeq
for $t > T^0$. Here the positive constant $Q$ is still the same as in \eqref{Q},  
$$
	Q = a_1 E_{N_a} + g_K (1 + 2H) ( |E_K| + 1) + 6 \left[\frac{a_1^2}{a_2} + \frac{ a_2}{E_{N_a}^2} + \frac{1}{a_2} g_K^2 (1 + 2H)^2\right].
$$
From the inequality \eqref{ky}, we see  
\beq \bl{WR}
	\frac{1}{2} \dca (W_{ij}^2 + \Pi_{ij}^2) \leq - \min \left\{ \frac{1}{2\tau_k}, \left[a_0 - \frac{k}{2\gb} + \frac{1}{2\tau_K} \gl^2 H^2 + n P - Q \right] \right\} (W_{ij}^2 +  \Pi_{ij}^2)
\eeq
for $ t > T_0$. If the threshold condition \eqref{Pnew} in this theorem is satisfied so that  
$$
	a_0 + \frac{1}{2\tau_K} \gl^2 H^2 + n P > Q + \frac{k}{2\gb} ,
$$
then \eqref{WR} is simply 
\beq  \bl{DG}
	\frac{1}{2} \dca (W_{ij}^2 + \Pi_{ij}^2) \leq - \delta (P) (W_{ij}^2 + \Pi_{ij}^2), \quad  t > T_0,
\eeq
and $\delta (P)$ given in \eqref{rt} is a positive constants.

Finally we can apply the fractional Gronwall inequality \eqref{pq}-\eqref{fG} to solve this differential inequality \eqref{DG} with respect to $W_{ij}^2(t) + \Pi_{ij}^2(t)$ and reach the conclusion that for any initial state $(V^0, R^0) \in \mathbb{R}^{2n}$, all the gap functions $W_{ij}(t) = V_i(t) - V_j(t)$ and $\Pi_{ij}(t) = R_i(t) - R_j(t)$ as $t \to \infty$ converge to zero at an asymptotic but non-exponential decaying rate $\mu_{\ap} (P, t)$ shown in \eqref{map}. Namely, for any initial state $(V^0, R^0) \in \mathbb{R}^{2n}$ and any $1 \leq i \neq j \leq n$,
\beq \bl{cov}
	\begin{split}
	\|W_{ij}(t) &\, \|^2 + \|\Pi_{ij}(t)\|^2 = \|V_i(t) - V_j(t)\|^2 + \|R_i(t) - R_j(t)\|^2   \\[2pt]
	\leq &\, \frac{\G(1 + \ap)}{\G (1 +\ap) + 2\delta (P) t^{\ap}} \left(\|V_i^0 - V_j^0\|^2 + \|R_i^0 - R_j^0\|^2 \right)   \\[3pt]
	= &\, \mu_{\ap}(P, t) \left(\|V_i^0 - V_j^0\|^2 + \|R_i^0 - R_j^0\|^2 \right) \to 0, \;\; \text{as} \;\; t \to \infty.
	\end{split}
\eeq
According to Definition \eqref{Def}, the fractional memristive Hodgkin-Huxley-Wilson neural network $\mathcal{MW}$ is complete synchronized at a fractional-power but non-exponential convergence rate. The proof is completed. 
\end{proof} 

As a remark, the memristor window function $\psi (s) = s(1 - \gb s)$ adopted in this model \eqref{Feq} of fractional memristive Hodgkin-Huxley-Wilson neural networks is for the typical Titanium dioxide memristors including the CMOS memristors (complementary metal-oxide-semiconductor circuits). The proof in this section can certainly be adaptable to deal with other types of memristor window functions such as the piecewise linear functions, hyperbolic-tangent functions, or Strukov-Williams memristors, etc.   

\section{\textbf{Conclusions}}

Hodgkin-Huxley equations  is the primary mathematical model of biological neurons and this model laid down a theoretical foundation for neuroscience as well as the modern research on biologically inspired neural networks. However, from mathematical point of view, the global and asymptotic dynamics of this higher dimensional and strongly nonlinear model has not been deeply investigated in comparison with several simplified neuron and neural network models like FitzHigh-Nagumo equations and Hindmarsh-Rose equations.

In this paper we introduced and investigated a new model of biological neural networks called Hodgkin-Huxley-Wilson neural networks. This model captured the essential features of the original Hodgkin-Huxley equations in terms of the quadratic nonlinearity as approximation of the exponential nonlinearity as well as the conductance coupling of two key ionic gating channels. The analytic proof methodology and techniques in this work can be adapted and generalized to explore future models of neuron networks with even higher order polynomial nonlinearity and with multiple ionic or other dynamic parameter channels. 

Asymptotic synchronization is an important topic about cognitive functionality and signals processing for biological neural networks. Synchronization is more diversified and complicated than the classical stability properties. It also encompasses chaotic synchronization and various types of approximate synchronization \cite{YC}. 

The main results in this paper are Theorem \ref{T2} and Theorem \ref{ThM}. The contribution of this work is the established explicit expression of the synchronization threshold and the exponential convergence rate in terms of the original model parameters, together with the purely analytic proof methodology and techniques by the leverage of hard analysis of differential inequalities. 

Furthermore we have extended the main result to Caputo fractional memristive Hodgkin-Huxley-Wilson neural networks in Theorem \ref{T4} through fractional Gronwall inequality. The synchronization threshold condition in this fractional case features a non-exponential convergence rate. 

Any new model with a new theory has its limitation. It is conceivable and promising that one can further pursue the exploration of synchronization based on Hodgkin-Huxley-Wilson neural networks in following two directions. It is challenging but one can apply the quadratic nonlinear approximation instead the sigmoidal function for the recovery ionic channel equation in an extended model. Moreover one can generalize Hodgkin-Huxley-Wilson neural networks from the ODE model to a PDE model consisting of diffusive membrane potential equation coupled with multiple ionic channel equations, which is biologically justified at least in one-dimensional distributive space because the neuron cells have long axon branches.  

\vspace{4pt}
\textbf{Funding}: This research received no external funding.

\vspace{4pt}
\textbf{Data Availability Statement}: The original contributions in this study are presented in the article. Further inquiries can be directed to the author.

\vspace{4pt}
\textbf{Conflicts of Interest}: The author declares no conflicts of interest.


\begin{thebibliography}{99}

\bibitem{HH}
A.L. Hodgkin, A.F. Huxley, \emph{A quantitative description of membrane current and its application to conduction and excitation in nerve}, The Journal of Physiology, \textbf{117}(4) (1952), 500-544.

\bibitem{HH1}
A.L. Hodgkin,  A.F. Huxley, \emph{Currents carried by sodium and potassium ions through the membrane of the giant axon of Loligo}, The Journal of Physiology,  \textbf{116}(4) (1952), 449-472.

\bibitem{HH2}
A.L. Hodgkin, A.F. Huxley, \emph{The components of membrane conductance in the giant axon of Loligo}, The Journal of Physiology,  \textbf{116}(4) (1952), 473-496.

\bibitem{HH3}
A.L. Hodgkin,  A.F. Huxley, \emph{The dual effect of membrane potential on sodium conductance in the giant axon of Loligo}, The Journal of Physiology,  \textbf{116}(4) (1952), 497-506.

\bibitem{HHB}
A.L. Hodgkin, A.F. Huxley, B. Katz, \emph{Measurement of current-voltage relations in the membrane of the giant axon of Loligo}, The Journal of Physiology,  \textbf{116}(4) (1952), 424-448.

\bibitem{KS}
J. Keener, J. Sneyd, \emph{Mathematical Physiology}, Springer-Verlag, New York, 1998.

\bibitem{ET}
G.B. Ermentrout, D.H. Terman, \emph{Mathematical Foundation of Neuroscience}, Springer, New York, 2010.

\bibitem{Wi}
H.R. Wilson, \emph{Simplified dynamics of human and mammalian neocortical neurons}, J. Theoretical Biology, \textbf{200} (1999), 375-388.

\bibitem{FHN}
R. FitzHugh, \emph{Impulses and physiological states in theoretical model of nerve membrane}, Biophysics Journal, \textbf{1} (1961), 445-466.

\bibitem{HR}
J.L. Hindmarsh, R.M. Rose, \emph{A model of neuronal bursting using three coupled first-order differential equations}, Proceedings of the Royal Society, London, Series B: Biological Sciiences, \textbf{221} (1984), 87-102.

\bibitem{ML}
C. Morris, H. Lecar, \emph{Voltage oscillations in the barnacle giant muscle fiber}, Biophysics Journal, \textbf{35}(1) (1981), 193.

\bibitem{Zh}
E.M. Izjikevich, \emph{Simple model of spiking neurons}, IEEE Trans. Neural Networks, \textbf{14} (2003), 1569-1572.

\bibitem{Cr}
J. Cronin, \emph{Mathematical Aspects of Hodgkin-Huxley Neural Theory}, Cambridge Univ. Press, Cambridge, UK, 1987. 

\bibitem{Ev}
J. Evans, \emph{Nerve axon equations IV: the stable and the unstable impulse}, Indiana Univ. Math. Journal, \textbf{24} (1975), 1160-1190.

\bibitem{FY}
W.E. Fitzgibbon, M. Parrot, Y. You, \emph{Finite dimensionality and upper semicontinuity of the global attractor of singularly perturbed Hodgkin-Huxley systems}, Journal of Differential Equations, \textbf{129}(1) (1996), 193-237. 

\bibitem{JJ}
D. Jaeger, R. Jung, \emph{Encyclopedia of Computational Neuroscience}, Springer, New York, NY, 2002. 

\bibitem{GO}
J. Guckenheimer, R.A. Oliva, \emph{Chaos in the Hodgkin-Huxley model}, SIAM Journal on Applied Dynamical Systems, \textbf{1}(1) (2002), 105-114.

\bibitem{WCF}
J. Wang, L. Chen, X. Fei, \emph{Analysis and control of the bifurcation of Hodgkin-Huxley model}, Chaos, Solitons and Fractals, \textbf{31}(1) (2007), 247-256.

\bibitem{BVL}
C.A. Batista, R.L. Viana, S.R. Lopes, A.M. Batista, \emph{Dynamic range in small-world networks of Hodgkin-Huxley neurons with chemical synapses}. Physica A: Statistical Mechanics and Its Applications, \textbf{410} (2014), 628-640.

\bibitem{LSY}
L. Skrzypek, Y. You, \emph{Dynamics and synchronization of boundary coupled FitzHugh-Nagumo neural networks}, Applied Mathematics and Computation, \textbf{388} (2020), 125545.

\bibitem{PYS}
C. Phan, Y. You, J. Su, \emph{Global dynamics of partly diffusive Hindmarsh-Rose equations in neurodynamics}, Dynamics of Partial Differential Equations, \textbf{18}(1) (2021), 33-47.

\bibitem{CY3}
C. Phan, Y. You, \emph{Random attractor for stochastic Hindmarsh-Rose equations with. additive noise}, Journal of Dynamics and Differential Equations, \textbf{33} (2021), 489-510.

\bibitem{AB}
M. Abeles, Y. Prut, H. Bergman, E. Vaadia, \emph{Synchronization in neuronal transmission and its importance for information processing}, Progress in Brain Research, \textbf{102} (1994), 395-404.

\bibitem{BP}
S. Boccaletti, A.N. Pisarchik, C.I. Del Genio, A. Amann, \emph{Synchronization From Coupled Systems to Complex Networks}, Cambridge University Press, Cambridge, UK, 2018.

\bibitem{HJ}
I. Hussain, S. Jafari, D. Ghosh, M. Perc, \emph{Synchronization and chimeras in a network of photosensitive FitzHugh-Nagumo neurons}, Nonlinear Dynamics, \textbf{104} (2021),  2711-2721.

\bibitem{WD}
Y. Wu et al, \emph{Dynamic learning of synchronization in coupled nonlinear systems}, Nonlinear Dynamics, https://doi.org/10.1007/s11071-024-10192-y, 2024. 

\bibitem{Y1}
Y. You, \emph{Global dynamics of diffusive Hindmarsh-Rose equations with memristors}, Nonlinear Analysis: Real World Applications, \textbf{71} (2023), 103827.

\bibitem{Y2}
Y. You, \emph{Exponential Synchronization of Memristive Hindmarsh-Rose Neural Networks}, Nonlinear Analysis: Real World Applications, \textbf{73} (2023), 103909.

\bibitem{YTT}
Y. You, J. Tian, J. Tu, \emph{Synchronization of memristive FitzHugh-Nagumo neural networks}, Communications in Nonlinear Science and Numerical Simulation, \textbf{125} (2023), 107405.

\bibitem{YT}
Y. You and J. Tu, \emph{Dynamics and synchronization of weakly coupled memristive reaction-diffusion neural networks}, Dynamics of Partial Differential Equations, \textbf{22} (2025), 1-28.

\bibitem{CPF}
O. Cornejo-Perez, R. Femat, \emph{Unidirectional synchronization of Hodgkin-Huxley neurons}, Chaos, Solitons and Fractals, \textbf{25} (2005), 43-53.

\bibitem{BB}
C. Batista, R.L. Viana, F. Ferrari, S.R. Lopes, A.M. Batista, J. Coninck, \emph{Control of bursting synchronization in networks of Hodgkin-Huxley-type neurons with chemical synapses}, Physical Review E, \textbf{87} (2013), 042713.

\bibitem{PL}
T. Prado, S.R. Lopes, C. Batista, J. Kurths, R.L. Viana, \emph{Synchronization of bursting Hodgkin-Huxley-type neurons in clustered networks}, Physical Review E, \textbf{90} (2014), 032818.

\bibitem{BAA}
S.Y. Bonabi, H. Asgharian, S. Safari, M.N. Ahmadabadi, \emph{EPGA implementation of a biological neural network based on the Hodgkin-Huxley neuron model}, Frontiers in Neuroscience, \textbf{8} (2014), 00379.

\bibitem{AO}
B. Ambrosio, M.A. Aziz-Alaoui, A. Oujbara, \emph{Synchronization in a three level network of all-to-all periodically forced Hodgkin-Huxley reaction-diffusion equations}, Mathematics, \textbf{12}(9) (2024), 1382. https://doi.org/10.3390/math12091382.

\bibitem{SS}
S. Saghafi, P. Sanaei, \emph{Dynamic entrainment: A deep learning and data-driven process approach for synchronization in the Hodgkin-Huxley model}, Chaos, \textbf{34}(10) (2024), 103124. 

\bibitem{WLZ}
L. Wei, D. Li, J. Zhang, \emph{Dynamics and synchronization of the Morris-Lecar model with field coupling subject to electromagnetic excitation}, Communications in Nonlinear Science and Numerical Simulation, \textbf{140} (2025), 108457.

\bibitem{Qi}
Y. Qi, A.L. Watts, J.W. Kim, P.A. Robinson, \emph{Firing patterns in a conductance-based neuron model: bifurcation, phase diagram, and chaos}, Biological Cybernetics, \textbf{107} (2013), 15-24.

\bibitem{VMN}
G. Vivekanandan, M. Mehrabbeik, H. Natiq, K. Rajagopal, E. Tlelo-Cuautle, \emph{Fractional-order memristive Wilson neuron model: dynamical analysis and synchronization patterns}, Mathematics, \textbf{10} (2022), 2827.

\bibitem{XJ}
Q. Xu, Z. Ju, S. Ding, C. Feng, M. Chen, B. Bao, \emph{Electromagnetic induction effects on electrical activity within a memristive Wilson neuron model}, Cognitive Neurodunamics, \textbf{16} (2022), 1221-1231.

\bibitem{Rz}
J. Rinzel, \emph{Excitation dynamics: insights from simplified membrane models}, Federation Proceedings, \textbf{44} (1985), 2944-2946.

\bibitem{CV}
V.V. Chepyzhov, M.I. Vishik, \emph{Attractors for Equations of Mathematical Physics}, American Mathematical Society, Providence, RI, 2002.

\bibitem{YC} 
Y. You, Approximate synchronization of memristive Hopfield neural networks, Axioms, \textbf{15}(3) (2026), 185. 

\bibitem{BJ}
B. Jin. \emph{Fractional Differential Equations -An Approach via Fractional Derivatives}, Springer, Cham, Switzerland, 2021.

\bibitem{XY}  
G. Xu, L. Yao, Z. Li, \emph{Review of the research of spiking neuron network based on memristors}, J. Biomedical Engineering, \textbf{35} (2018), 475-480. 

\bibitem{KD}
K. Diethelm, \emph{The Analysis of Fractional Differential Equations}, Springer-Verlag, Berlin, 2010.

\bibitem{YY}
Y. You, \emph{Robust synchronization of time-fractional memristive Hopfield neural networks}, Axioms, \textbf{15}(1) (2026), 37.

\bibitem{DK} 
T.S. Doan, P.E. Kloeden, \emph{Semi-dynamical systems generated by autonomous Caputo fractional Differential equations}, Vietnam Journal of Mathematics, \text{49} (2021), 1305-1315. 

\bibitem{PK} 
P.E. Kloeden, \emph{An elementary inequality for dissipative Caputo fractional differential equations}, Fractional Calculus and Applied Analysis, \textbf{26} (2023), 2166-2174.

\end{thebibliography}
\end{document}